\documentclass[12pt]{amsart}

\usepackage{amssymb, graphicx, labelfig, hyperref}
\usepackage[all]{xy}

\newcommand{\Z}{\mathbb Z}
\newcommand{\C}{\mathbb C}
\newcommand{\R}{\mathbb R}

\newtheorem{thm}{Theorem}
\newtheorem{prop}[thm]{Proposition}
\newtheorem{lem}[thm]{Lemma}
\newtheorem{cor}[thm]{Corollary}
\newtheorem{con}[thm]{Conjecture}

\theoremstyle{remark}
\newtheorem{rem}[thm]{Remark}

\newcommand{\Oq}{\mathcal O_q}

\newcommand{\SLd}{\mathrm {SL}_d}
\newcommand{\SL}{\mathrm {SL}}
\newcommand{\GLd}{\mathrm {GL}_d}
\newcommand{\GL}{\mathrm {GL}}
\newcommand{\SSS}{\mathcal S^q_{\SLd}}

\newcommand{\SLtwo}{\mathrm {SL}_2}
\newcommand{\SSStwo}{\mathcal S^q_{\SLtwo}}
\newcommand{\SLthree}{\mathrm {SL}_3}
\newcommand{\SSSthree}{\mathcal S^q_{\SLthree}}

\newcommand{\Aio}{\SSS(A)_{\mathrm{io}}}
\newcommand{\Tr}{\mathop{\mathrm{Tr}}}

\renewcommand{\leq}{\leqslant}
\renewcommand{\geq}{\geqslant}
\renewcommand{\phi}{\varphi}

\newcommand{\qint}[1]{\left[#1\right]_{q}}
\newcommand{\qfact}[1]{\left[#1\right]_{q}!}
\newcommand{\qbinom}[2]{{\genfrac{[}{]}{0pt}{}{#1}{#2}}_q}

\title[Central elements in the skein algebra]{Central elements in the $\SLd$--skein algebra 
\\ of a surface}
\author{Francis Bonahon}
\address{Department of Mathematics, University of Southern California, Los Angeles CA 90089-2532, U.S.A.}
\email{fbonahon@usc.edu}
\address{Department of Mathematics, Michigan State University, East Lansing MI 48824, U.S.A.}
\email{bonahonf@msu.edu}
\author{Vijay Higgins}
\address{Department of Mathematics, Michigan State University, East Lansing MI 48824, U.S.A.}
\email{higgi231@msu.edu}

\date{\today}

\thanks{This work was developed under the auspices of  the Research Training Grant DMS-2135960, \emph{RTG: Algebraic and Geometric Topology at Michigan State}, from the U.S. National Science Foundation.}

\begin{document}

 \maketitle

\begin{abstract}
The $\SLd$--skein algebra $\SSS(S)$ of a surface $S$ is a certain deformation of the coordinate ring of the character variety consisting of flat $\SLd$--local systems over the surface. As a quantum topological object, $\SSS(S)$ is also closely related to the HOMFLYPT polynomial invariant of knots and links in $\R^3$. We exhibit a very rich family of central elements in this algebra $\SSS(S)$ that appear when the quantum parameter $q$ is a root of unity. These central elements are obtained by threading along framed links certain polynomials arising in the elementary theory of symmetric functions, and related to taking powers in $\SLd$. 
 
\end{abstract}
 
 Let $\GLd$ denote the general linear group of invertible $d$-by-$d$ matrices, and let the special linear group $\SLd$ consist of those matrices that have determinant 1.
 The \emph{$\SLd$--skein module} $\SSS(M)$ of an oriented 3-dimensional manifold $M$ is a certain deformation of the coordinate ring of the character variety
 $$
 \mathcal X_{\SLd} (M) = \{ \text{homomorphisms } r \colon \pi_1(M) \to \SLd \} /\kern -3pt/\GLd,
 $$
 where $\GLd$ acts on the set of homomorphisms $r \colon \pi_1(M) \to \SLd $ by conjugation. This quantum deformation depends on a quantum parameter $q$, and more precisely on a $d$--root $q^{\frac1d}$. In its current incarnations, the motivation for this mathematical object arises from Witten's topological quantum field theory interpretation of the Jones polynomials and other knot invariants \cite{Wit}, where the elements of $\SSS(M)$ occur as morphisms. In particular, it is closely related to the HOMFLYPT invariant of knots and links in $\R^3$ \cite{Yok}. 
 
 The elements of  $\SSS(M)$ can be represented by linear combinations of framed links in $M$ where each component is colored by an integer $i \in \{1,2,\dots, d-1\}$, standing for the $i$--th exterior power of the defining representation of the quantum group $\mathrm U_q(\mathfrak{sl}_d)$. The relations satisfied by these generators correspond to the full set of relations between tensor products of these representations in the braided tensor category of representations of $\mathrm U_q(\mathfrak{sl}_d)$ \cite{CKM}. When $d>2$, these relations are better expressed in terms of more complicated objects called $\SL_d$--webs; see \S \ref{sect:SkeinModule}. 
 
 A special case of interest is the one where $M$ is equal to the thickening $S\times [0,1]$ of an oriented surface $S$ of finite topological type, in which case the resulting skein module $\SSS(S) = \SSS \big( S \times [0,1] \big)$ is endowed with a natural multiplication by superposition; see \S \ref{sect:SkeinAlgebra}. The viewpoint of \cite{Wit} involves representations of this algebra $\SSS(S)$ and, if we want these representations to have finite dimension, it is natural to require that the quantum parameter $q$ be a root of unity. 
 
 In the special case where $d=2$ and $q$ is a primitive $n$--root of unity with $n$ odd, Helen Wong and the first author \cite{BonWon} discovered unexpected central elements in the skein algebra $\SSStwo(S)$, based on the Chebyshev polynomial of the first type $T_n \in \Z[e]$; see \cite{Le1} for versions when $n$ is even. Frohman, Kania-Bartoszy\'nska and L\^e \cite{FroKanBarLe}  later proved that these elements, together with the more obvious central elements associated to punctures that occur for all $q$, generate the whole center of $\SSS(S)$. This, together with a combination of results from \cite{BonWon, FroKanBarLe, GanJorSaf}, led to a classification of ``most'' finite-dimensional representations of $\SSStwo(S)$, in terms of points in a certain finite branched cover of the character variety $ \mathcal X_{\mathrm{SL}_2} (M) $. 
 
 The current article is devoted to the development of similar central elements in the $\SLd$--skein algebra $\SSS(S)$, when $q$ is a root of unity. 
 
 In the case of $\SLtwo$, the regular functions on $\SLtwo$ that are invariant under conjugation by elements of $\mathrm{GL}_2$ form a polynomial algebra generated by the trace function $\Tr$, and the Chebyshev polynomial $T_n$ is determined by the property that $\Tr A^n = T_n(\Tr A)$ for every $A\in \mathrm{SL}_2$. For $\SLd$, the algebra of $\GLd$--invariant regular functions is a polynomial algebra in $d-1$ variables, corresponding to the elementary symmetric polynomials $E_d^{(1)}$, $E_d^{(2)}$, \dots, $E_d^{(d-1)}$ in the eigenvalues. These are also related to the coefficients of the characteristic polynomial by the property that
 $$
 \det(A+ t\, \mathrm{Id}_d ) =  t^d + t^{d-1} E_d^{(1)}(A) +  t^{d-2} E_d^{(2)}(A) + \dots   + t E_d^{(d-1)}(A) + 1
 $$
 for every $A\in \SLd$. An immediate consequence of the elementary theory of symmetric functions  is that, for every $n\geq 1$ and for every $i \in \{1,2,\dots, d-1\}$, there is a unique polynomial $\widehat P_d^{(n,i)} \in \Z[e_1, e_2, \dots, e_{d-1}]$ such that
 $$
 E_d^{(i)}(A^n) = \widehat P_d^{(n,i)} \big( E_d^{(1)}(A), E_d^{(2)}(A), \dots, E_d^{(d-1)}(A) \big)
 $$
 for every $A\in \SLd$; see \S \ref{sect:PowerPols}. We call these polynomials $\widehat P_d^{(n,i)} $ the \emph{reduced power elementary polynomials}. For instance, when $d=2$, there is only one such polynomial $\widehat P_2^{(n,1)}$ for every $n$, and this polynomial is just the Chebyshev polynomial $T_n$. 
 
 Our new central elements in $\SSS(S)$ are based on the threading operation along polynomials that was already  at the basis of \cite{BonWon}. For a  framed knot $L$ in a 3-manifold $M$, the threading along a polynomial 
  $$
 P = \sum_{i_1, i_2, \dots, i_{d-1}=0}^{i_{\mathrm{max}}} a_{i_1i_2 \dots i_{d-1}} e_1^{i_1}e_2^{i_2} \dots e_{d-1}^{i_{d-1}}  \in \Z[e_1, e_2, \dots, e_{d-1}]
 $$
 associates to $L$ the skein
 $$
 L^{[P]} =  \sum_{i_1, i_2, \dots, i_{d-1}=0}^{i_{\mathrm{max}}} a_{i_1i_2 \dots i_{d-1}}  L^{[ e_1^{i_1}e_2^{i_2} \dots e_{d-1}^{i_{d-1}}]} \in \SSS(M),
 $$
 where
 $ L^{[ e_1^{i_1}e_2^{i_2} \dots e_{d-1}^{i_{d-1}}]} \in \SSS(M)$ is represented by the union of $i_1+i_2+\dots i_{d-1}$ disjoint parallel copies of the knot $L$, taken in the direction of the framing, and with $i_1$ of these copies carrying the weight $1$, $i_2$ carrying the weight $2$, \dots, and $i_{d-1}$ carrying the weight $d-1$. A similar construction applies to links $L$ with several components. See \S \ref{sect:Thread} for details.

\begin{thm}
\label{thm:PowerElementaryThreadingCentralIntro}
 Suppose that the $d$--root $q^{\frac 1d}$ occurring in the definition of skein modules $\SSS(M)$ is such that $q^{\frac {2n}d}=1$, and that $q^{2i}\neq 1$ for every integer $i$ with $2\leq i \leq \frac d2$. In a thickened surface $S \times [0,1]$, let $L=L_1 \sqcup L_2 \sqcup \dots \sqcup L_c$ be a framed link in which each component $L_j$ carries a weight $i_j \in \{1,2, \dots, d-1\}$. Then the skein $L^{[\widehat P_d^{(n, \bullet)}]} \in \SSS(S)$ obtained by threading the reduced power elementary polynomial $\widehat P_d^{(n, i_j)} \in \Z[e_1, e_2, \dots, e_{d-1}]$ along each component $L_j$ is central in the skein algebra $\SSS(S)$ of the surface $S$.  
 \end{thm}

 Theorem~\ref{thm:PowerElementaryThreadingCentralIntro} is based on a more general property for skein modules $\SSS(M)$ of $3$--manifolds which, borrowing terminology from \cite{Le1}, is a certain transparency property for  threading operations along the reduced power  polynomial $\widehat P_d^{(n, i)}\in \Z[e_1, e_2, \dots, e_{d-1}]$. This property states that, if $L_0$ is a framed link in a 3--manifold $M$ carrying component weights in $\{ 1,2, \dots, d-1\}$ and if $L$ is a framed knot disjoint from $L_0$, then the skein $L_0 \sqcup L^{[P_d^{(n,i)}]} \in \SSS(M)$ obtained by threading   $\widehat P_d^{(n,i)} $ along $L$ is invariant under any isotopy of $L$ in $M$ that is allowed to cross $L_0$. 

\begin{thm}
\label{thm:PowerElementaryThreadingTransparentIntro}
 Suppose that the $d$--root $q^{\frac 1d}$ occurring in the definition of skein modules $\SSS(M)$ is such that $q^{\frac {2n}d}=1$, and that $q^{2i}\neq 1$ for every integer $i$ with $2\leq i \leq \frac d2$. Then, for every $i=1$, $2$, \dots, $d-1$, the threading operation along the reduced power elementary  polynomial  $\widehat P_d^{(n,i)} \in \Z[e_1, e_2, \dots, e_{d-1}]$ is transparent in the skein module $\SSS(M)$ of any oriented $3$--manifold $M$. 
\end{thm}

As indicated in Remark~\ref{rem:ConditionsOnQ}, the hypothesis in Theorems~\ref{thm:PowerElementaryThreadingCentralIntro} and \ref{thm:PowerElementaryThreadingTransparentIntro} that $q^{2i}\neq 1$ for every integer $i$ with $2\leq i \leq \frac d2$ is probably unnecessary. 

Similar results for $\mathrm G_2$--skeins,  where $\mathrm G_2$ is the exceptional Lie group of rank 2, will appear in \cite{BBHHMP}.

  \section{$\SLd$--webs and skein relations}

 \subsection{The $\SLd$--skein module of a $3$--dimensional manifold}
 \label{sect:SkeinModule}  Throughout the article, $\SLd$ will denote the Lie group of $d$-by-$d$ matrices with determinant~1. Because the coefficient ring of this algebraic group is  irrelevant for our purposes, we will systematically omit it.

 We are here using the version of $\SLd$--skein modules that uses the webs developed by Cautis-Kamnitzer-Morrison in \cite{Mor, CKM}. There is another well-known alternative based on Kuperberg-Sikora spiders \cite{Kup, Sik, LeSik}. See \cite{Pou} for the equivalence between the two viewpoints.

 An \emph{$\SLd$--web} in an oriented  3-dimensional manifold $M$ is a graph $W$ embedded in $M$ endowed with additional data satisfying the following conditions: 
\begin{enumerate}
\item the graph $W$ is endowed with a \emph{ribbon structure} consisting of a thin oriented surface embedded in $M$ that contains $W$ and deformation retracts onto it;
 \item each edge of $W$ carries an orientation and a weight $i \in \{1,2, \dots, d-1\}$;
 \item each vertex of $W$ is of one of the following three types:
\begin{enumerate}
\item a vertex of type ``merge'' with two incoming edges of weights $i$ and $j$ and one outgoing edge of weight $i+j$, as in the first picture of Figure~\ref{fig:Vertices};
\item a  vertex of type ``split'' with one incoming edge of weight $i+j$ and two outgoing edges of weights $i$ and $j$, as in the second picture of Figure~\ref{fig:Vertices};
\item a vertex of type ``stump'' (also called ``tag'' in \cite{CKM}) adjacent to exactly one edge of $W$, which carries weight $d$, as in the last two pictures of Figure~\ref{fig:Vertices};
\end{enumerate}

\item the only edges that are allowed to carry weight $d$ are those adjacent to a stump;
\item $W$ can have components that are closed loops, with no vertices, but no component can be the graph with exactly one edge and two stumps.
\end{enumerate}

Along the components of $W$ that are closed loops, the ribbon structure is equivalent to the very classical notion of \emph{framing}, namely the data of a vector field that is everywhere transverse to the loop (or, equivalently, with a trivialization of the normal bundle of that loop). In particular, framed (oriented) links where each component is colored by a weight $i \in \{1,2, \dots, d-1\}$ are fundamental examples of webs.

 \begin{figure}[htbp]
  
\centerline{\SetLabels
\R( .2 * .3 )  $i$ \\
\L( .85 * .3)  $j$ \\
\L(  .6* .7 )  $i+j$ \\
(0.5 * -.4)  merge vertex \\
\endSetLabels
\AffixLabels{\includegraphics{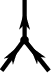}}
\hskip 50pt
\SetLabels
\R( .2 * .3 )  $i$ \\
\L( .85 * .3)  $j$ \\
\L(  .65* .7 )  $i+j$ \\
(0.5 * -.4)  split vertex \\
\endSetLabels
{\AffixLabels{ \includegraphics{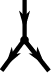}}}
\hskip 50pt
\SetLabels
\R( .3 * .25 )  $i$ \\
\L( .8 * .25)  $d-i$ \\
\L(  .6* .7 )  $d$ \\
(0.5 * -.3)  outward stump \\
\endSetLabels
{\AffixLabels{ \includegraphics{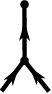}}}
\hskip 50pt
\SetLabels
\R( .3 * .25 )  $i$ \\
\L( .8 * .25)  $d-i$ \\
\L(  .65* .7 )  $d$ \\
(0.5 * -.3)  inward stump \\
\endSetLabels
{\AffixLabels{ \includegraphics{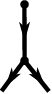}}}
}
 \vskip 20 pt
\caption{Vertices of a web}
\label{fig:Vertices}
\end{figure}
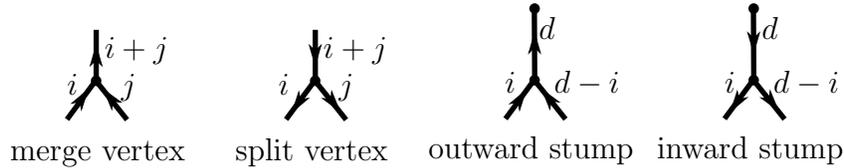

The \emph{$\SLd$--skein module} $\SSS(M)$ is  obtained from the vector space over $\C$ (say) freely generated by the set of isotopy classes of $\SLd$--webs under a set of \emph{skein relations} that are explicitly listed in \cite{CKM}. Since we will not need most of them, we are only listing a few in Figures~\ref{fig:Skein1}--\ref{fig:Braiding} and refer to \cite{CKM} for the full list.

\begin{figure}[htbp]

\begin{align*}
\SetLabels\small
\R\E( 0 * .1 ) $i$  \\
\L\E( 1 * .1 )  $j$ \\
(  .5* .2 ) $k$  \\
( .5 * .8 ) $l$  \\
\R\E(  0* .5 )  $i-k$ \\
\L\E( 1 * .5 ) $k+j$  \\
\R\E( 0 * .9 ) $i-k+l$  \\
\L\E( 1 * .9 )  $j+k-l$ \\
\endSetLabels
{\AffixLabels{\raisebox{-30pt}{\includegraphics{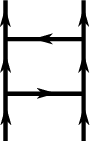}}}}
 \hskip 3em 
&= \sum_m \qbinom{i-j-k+l}{m}
\hskip 4em   
\SetLabels\small
\R\E( 0 * .1 ) $i$  \\
\L\E( 1 * .1 )  $j$ \\
(  .5* .2 ) $l-m$  \\
( .5 * .8 ) $k-m$  \\
\R\E(  0* .5 )  $i+l-m$ \\
\L\E( 1 * .5 ) $j+l-m$  \\
\R\E( 0 * .9 ) $i-k+l$  \\
\L\E( 1 * .9 )  $j+k-l$ \\
\endSetLabels
{\AffixLabels{\raisebox{-30pt}{\includegraphics{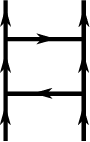}}}}
\end{align*}

\caption{A typical skein relation}
\label{fig:Skein1}
\end{figure}

\begin{figure}[htbp]

 \begin{align*}
 \SetLabels
\E\R( -.0 *  .5)  $i$ \\
\E\R( .85 * .5 )   $j$\\
\R( .4 * .8 )   $i+j$\\
\R(  .4* .1 )   $i+j$\\
\endSetLabels\small
{\AffixLabels{\raisebox{-30pt}{\includegraphics{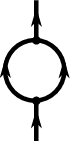}}}}
 &= \qbinom{i+j}{i}\   
\SetLabels
\L(  1* .8 )   $i+j$\\
\endSetLabels
{\AffixLabels{ \raisebox{-30pt}{\includegraphics{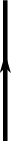}}}}
 \end{align*}

\caption{Another skein relation}
\label{fig:Skein2}
\end{figure}

\begin{figure}[htbp]
\begin{align*}
\raisebox{-15pt}{\SetLabels
\L(  1* .8 )  $i$ \\
\R( .8 * 0 )  $d-i$ \\
(  .5* .6 )  $d$ \\
\endSetLabels
{\AffixLabels{\includegraphics{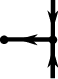}}}
}
&= 
(-1)^{i(d-i)}\kern 5pt
\raisebox{-15pt}{\SetLabels
\R(  0* .8 )  $i$ \\
\L( .2 * .1 )  $d-i$ \\
(  .5* .6 )  $d$ \\
\endSetLabels
{\AffixLabels{\includegraphics{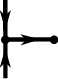}}}
}
&\raisebox{-25pt}{\SetLabels
\R(  .45* .0 )  $i$ \\
\L(  .55* 1 )  $i$ \\
\E\R( .45 * .5 )  $d-i$ \\
(  .25* .75 )  $d$ \\
(  .75* .4 )  $d$ \\
\endSetLabels
{\AffixLabels{\includegraphics{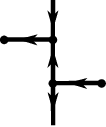}}}
}
&=
\raisebox{-25pt}{\SetLabels
\R(  0* .0 )  $i$ \\
\L(  1* 1 )  $i$ \\
\endSetLabels
{\AffixLabels{\includegraphics{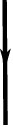}}}
}
\end{align*}

\caption{Two skein relations involving stumps}
\label{fig:Skein3}
\end{figure}

 \begin{figure}[htbp]
\begin{align*}
\SetLabels\small
\R\E( 0 * .1 ) $i$  \\
\L\E( 1 * .1 )  $j$ \\
\R\E( 0 * .9 ) $j$  \\
\L\E( 1 * .9 )  $i$ \\
\endSetLabels
{\AffixLabels{\raisebox{-30pt}{\includegraphics{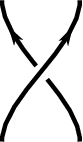}}}}
 \hskip 2em 
&= 
(-1)^{ij}q^{\frac{ij}d}
\sum_m (-q )^{-m}
\hskip 2em   
\SetLabels\small
\R\E( 0 * .1 ) $i$  \\
\L\E( 1 * .1 )  $j$ \\
(  .5* .2 ) $i-m$  \\
( .5 * .8 ) $j-m$  \\
\R\E(  0* .5 )  $m$ \\
\L\E( 1 * .5 ) $i+j-m$  \\
\R\E( 0 * .9 ) $j$  \\
\L\E( 1 * .9 )  $i$ \\
\endSetLabels
{\AffixLabels{\raisebox{-30pt}{\includegraphics{Skein6.eps}}}}
\\
\\
\SetLabels\small
\R\E( 0 * .1 ) $i$  \\
\L\E( 1 * .1 )  $j$ \\
\R\E( 0 * .9 ) $j$  \\
\L\E( 1 * .9 )  $i$ \\
\endSetLabels
{\AffixLabels{\raisebox{-30pt}{\includegraphics{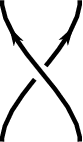}}}}
 \hskip 2em 
&= 
(-1)^{ij}q^{-\frac{ij}d}
\sum_m (-q )^{m}
\hskip 2em   
\SetLabels\small
\R\E( 0 * .1 ) $i$  \\
\L\E( 1 * .1 )  $j$ \\
(  .5* .2 ) $i-m$  \\
( .5 * .8 ) $j-m$  \\
\R\E(  0* .5 )  $m$ \\
\L\E( 1 * .5 ) $i+j-m$  \\
\R\E( 0 * .9 ) $j$  \\
\L\E( 1 * .9 )  $i$ \\
\endSetLabels
{\AffixLabels{\raisebox{-30pt}{\includegraphics{Skein6.eps}}}}
\end{align*}

\caption{Braiding relations}
\label{fig:Braiding}
\end{figure}

In these figures, each skein relation should be seen as occurring in a neighborhood of a disk embedded in $M$, in such a way that the ribbon structures of each web represented are horizontal for the projection to that disk. The sums are over indices $m\in \Z$, with the following conventions:
\begin{enumerate}
 \item the sum is limited to those values of $m$ that lead to edge weights in $\{0,1,\dots, d\}$;
 \item an edge carrying  weight $0$ and its end vertices should be erased;  
 \item an edge carrying weight $d$ should be split into two edges with stumps, with a convention that will be more precisely described when we need it in our proof of Lemma~\ref{lem:ITimesL_j}. 
 \end{enumerate}
 Also, the symbols $\qbinom ij$ represent the \emph{quantum binomials}
 $$
 \qbinom ij = \frac{\qint i \qint{i-1} \dots \qint{i-j+1}}{\qint j \qint{j-1}\dots  \qint 2 \qint 1} = \frac{\qfact i}{\qfact j \qfact{n-j}}
 $$
 with the \emph{quantum integers} 
 $$\qint i = \frac{q^i - q^{-i}}{q-q^{-1}}$$ and the \emph{quantum factorials} 
 $$\qfact i = \qint i \qint{i-1} \dots \qint2 \qint1. $$
 
We will not need the skein relation of Figure~\ref{fig:Skein1}, which is shown here only to give the flavor of typical skein relations. However, we will make use of the relations of Figures~\ref{fig:Skein2}--\ref{fig:Braiding}. 

Note that the \emph{braiding relations} of Figure~\ref{fig:Braiding}  require us to  fix a $d$--root $q^{\frac1d}$ of the quantum parameter $q \in \C-\{0\}$. As a consequence, the skein module $\SSS(M)$ depends on this choice of $q^{\frac1d}$ in spite of the fact that this is not reflected in the notation, which would otherwise be too cumbersome.

These skein relations originate from the representation theory of the quantum group $\mathrm U_q(\mathfrak{sl}_d)$. The skein relations other than the braiding relations of Figure~\ref{fig:Braiding} describe all the relations that occur between tensor products of the quantum exterior power representations $\bigwedge^i_q \C^d$ of $\mathrm U_q(\mathfrak{sl}_d)$. The braiding relations reflect the braiding of the representation category of  $\mathrm U_q(\mathfrak{sl}_d)$. See \cite{CKM} for details.

 \subsection{The $\SLd$--skein algebra of a surface}
 \label{sect:SkeinAlgebra}
 
 An important special case is provided by the thickening $M=S \times [0,1]$ of an oriented surface $S$. In this case the skein module $\SSS(S\times [0,1])$ admits a natural algebra structure where the multiplication is defined as follows. If $[W_1]$, $[W_2] \in \SSS(S)$ are respectively represented by webs $W_1$, $W_2$ in $S\times [0,1]$, the product $[W_1] \bullet [W_2] $ is represented by the web $W_1' \cup W_2'$ where $W_1'$ is obtained by rescaling $W_1$ inside $S\times [ 0,\frac12]$  and  $W_2'$ is obtained by rescaling $W_2$ inside $S\times [ \frac12, 1]$. In practice if, by projection to $S$, we represent each $W_i$ 
 by the picture of a possibly knotted graph in $S$, $[W_1] \bullet [W_2] $ is obtained by placing $W_2$ on top of $W_1$. 
 
 The algebra $\SSS(S\times [0,1])$, denoted as $\SSS(S)$ for short, is the \emph{$\SLd$--skein algebra} of the oriented surface $S$.

 \section{Threading a polynomial along a framed link}
 \label{sect:Thread}
 
 Let $L$ be an oriented framed knot in the 3--manifold $M$, namely a 1--dimensional oriented closed submanifold of $M$ that is endowed with a nonzero section of its normal bundle. This framing can also be used to define a ribbon structure along $L$.

Given a polynomial 
 $$
 P = \sum_{i_1, i_2, \dots, i_{d-1}=0}^{i_{\mathrm{max}}} a_{i_1i_2 \dots i_{d-1}} e_1^{i_1}e_2^{i_2} \dots e_{d-1}^{i_{d-1}}  \in \Z[e_1, e_2, \dots, e_{d-1}]
 $$
in $(d-1)$ variables $e_1$, $e_2$, \dots, $e_{d-1}$ with coefficients $a_{i_1i_2 \dots i_{d-1}}  \in \Z$, the \emph{skein obtained by threading $P$ along $L$} is defined as the linear combination
 $$
 L^{[P]} =  \sum_{i_1, i_2, \dots, i_{d-1}=0}^{i_{\mathrm{max}}} a_{i_1i_2 \dots i_{d-1}}  L^{[ e_1^{i_1}e_2^{i_2} \dots e_{d-1}^{i_{d-1}}]} \in \SSS(M),
 $$
 where
 $ L^{[ e_1^{i_1}e_2^{i_2} \dots e_{d-1}^{i_{d-1}}]} \in \SSS(M)$ is represented by the union of $i_1+i_2+\dots i_{d-1}$ disjoint parallel copies of the knot $L$, taken in the direction of the framing, and with $i_1$ of these copies carrying the weight $1$, $i_2$ carrying the weight $2$, \dots, and $i_{d-1}$ carrying the weight $d-1$. In particular, $L^{[ e_1^{0}e_2^{0} \dots e_{d-1}^{0}]} $ is represented by the empty link. 
 
 More generally, if $L$ is an oriented framed link with components $L_1$, $L_2$, \dots $L_c$, the \emph{skein obtained by threading the polynomials $P_j$ along the components $L_j$ of $L$} is defined as the disjoint union
 $$
 L^{[P_1, P_2, \dots, P_c]} = L_1^{[P_1]} \sqcup  L_2^{[P_2]} \sqcup \dots  \sqcup L_c^{[P_c]} 
 $$
 where the parallel copies used to define each $ L_j^{[P_j]} $ are chosen in disjoint tubular neighborhoods of the $L_j$. Note that, because each $ L_j^{[P_j]} $ is represented by a linear combination of webs, the disjoint union $ L^{[P_1, P_2, \dots, P_c]}  \in \SSS(M)$ is also defined by a linear combination of disjoint union of those webs. 
 
 Threading a polynomial $P \in \Z[e_1, e_2, \dots, e_{d-1}]$ is \emph{transparent} if, for every oriented framed knot $L$ in a 3--manifold $M$ and every web $W\subset M$ that is disjoint from $L$, the element of $\SSS(M)$ that is represented by $L^{[P]} \sqcup W$ is invariant under any isotopy of the knot $L$ that allows it to cross $W$.

\begin{lem}
\label{lem:TransparentGivesCentral}
 If threading a polynomial $P \in \Z[e_1, e_2, \dots, e_{d-1}]$ is transparent then, for every surface $S$ and every oriented framed link $L\subset S \times [0,1]$, the skein $L^{[P]}$ obtained by threading $P$ along each component of $L$ is central in the skein algebra $\SSS(S)$. 
\end{lem}
\begin{proof} If $[W] \in \SSS(S)$ is represented by a web $W \subset S \times [\frac13, \frac23]$, then $ [L^{[P]}] \bullet [W] $ is represented by $L_1^{[P]} \sqcup W$ where $L_1$ is obtained by rescaling $L$ inside $S\times [0, \frac13]$, while $[W] \bullet  [L^{[P]}]  $ is represented by $L_2^{[P]} \sqcup W$ with $L_2$  obtained by rescaling $L$ inside $S\times [\frac23, 1]$. Applying the transparency property to an isotopy moving $L_1$ to $L_2$ shows that  $ [L^{[P]}] \bullet [W] = [W] \bullet  [L^{[P]}]  $. 
\end{proof}

 \section{Power elementary  polynomials}
 \label{sect:PowerPols}
 
 In the ring  $\Z[\lambda_1, \lambda_2, \dots, \lambda_d]$ of polynomials with integer coefficients in $d$ variables $\lambda_1$, $\lambda_2$, \dots, $\lambda_d$, recall that a polynomial is \emph{symmetric} if it is invariant under all permutations of the variables $\lambda_1$, $\lambda_2$, \dots, $\lambda_d$. 
 Fundamental examples include the \emph{elementary symmetric polynomials}
 $$
 E_d^{(i)}= \sum_{1\leq j_1<j_2<\dots<j_i\leq d} \lambda_{j_1} \lambda_{j_2} \dots  \lambda_{j_i},
 $$ 
 defined for $1\leq i \leq d$. 
 
  There is a well-known connection between the elementary symmetric polynomials $E_d^{(i)}$ and the Lie group $\GLd$. Namely, if $A\in \GLd(\mathbb K)$ is a matrix with coefficients in the field $\mathbb K$, with eigenvalues $\lambda_1$, $\lambda_2$, \dots, $\lambda_d$ in the algebraic closure of $\mathbb K$, the coefficient of the term of degree $d-i$ in the characteristic polynomial of $A$ is equal to $(-1)^i E_d^{(i)}$. In this situation, we will also write 
  $$
  E_d^{(i)}(A) = E_d^{(i)}(\lambda_1, \lambda_2, \dots, \lambda_d)\in \mathbb K.
  $$
  
  If we are interested in the characteristic polynomial of a power $A^n$, whose eigenvalues are $\lambda_1^n$, $\lambda_2^n$, \dots, $\lambda_d^n$, it makes sense to consider, for $n\geq 1$ and $1\leq k \leq d$,  the \emph{power  elementary symmetric  polynomials}
 $$
 E_d^{(n,i)} = \sum_{1\leq j_1<j_2<\dots<j_i \leq d} \lambda_{j_1}^n \lambda_{j_2}^n \dots  \lambda_{j_i}^n 
 $$
 obtained from $E_d^{(i)}$ by replacing each occurrence of the variable $\lambda_j$ with its power $\lambda_j^n$.
 For instance, the case  $n=1$ gives the original elementary symmetric polynomial $E_d^{(1,i)}=E_d^{(i)}$, while the case $i=1$ corresponds to the well-known family of \emph{power sum polynomials} $E_d^{(n,1)} = \sum_{i=1}^d \lambda_i^n$.

\begin{lem}
\label{lem:PowerElementaryPolExists}
 There exists a unique polynomial $P_d^{(n,i)} \in \Z[e_1, e_2, \dots, e_d]$ such that $E_d^{(n,i)} \in \Z[\lambda_1, \lambda_2, \dots, \lambda_d]$ is obtained from $P_d^{(n,i)}$ by replacing each variable $e_j$ with the elementary symmetric polynomial $E_d^{(j)} \in  \Z[\lambda_1, \lambda_2, \dots, \lambda_d]$. 
\end{lem}

\begin{proof}
 This is an immediate consequence of the very classical property that the subring of symmetric polynomials in $\Z[\lambda_1, \lambda_2, \dots, \lambda_d]$  is itself isomorphic to the polynomial ring $\Z[e_1, e_2, \dots, e_d]$, by an isomorphism sending each elementary symmetric polynomial $E_d^{(i)}$ to the variable $e_i$. See for instance \cite[\S I.2]{McD}. 
\end{proof}

We call these $P_d^{(n,i)} \in \Z[e_1, e_2, \dots, e_d]$ the \emph{power elementary polynomials}, not to be confused with the closely connected but formally different power elementary \emph{symmetric} polynomials $E_d^{(n,i)} \in \Z[\lambda_1, \lambda_2, \dots, \lambda_d]$, which involve different variables. 

Simple considerations show that $P_d^{(1,i)} = e_i$ when $n=1$, and $P_d^{(n,d)}= e_d^n$ when $i=d$. More generally, the isomorphism between the ring of symmetric polynomials and the polynomial ring in the elementary symmetric polynomials is fairly algorithmic (see for instance Property (2.3) of \cite[\S I.2]{McD}) and the power elementary polynomials $P_d^{(n,i)} \in \Z[e_1, e_2, \dots, e_d]$  can be explicitly computed, although their sizes quickly grow with $d$ and $n$ and eventually require the use of mathematical software. For instance, when $d=4$ and $n=6$,
\begin{align*}
P_4^{(6,1)}&=   e_1^6-6 e_1^4 e_2+9 e_1^2 e_2^2-2 e_2^3+6 e_1^3 e_3-12 e_1 e_2 e_3+3 e_3^2-6 e_1^2 e_4+6 e_2 e_4
\\
P_4^{(6,2)}&=  e_2^6-6 e_1 e_2^4 e_3+9 e_1^2 e_2^2 e_3^2+6 e_2^3 e_3^2-2 e_1^3 e_3^3-12 e_1 e_2 e_3^3+3 e_3^4+6 e_1^2 e_2^3 e_4-6 e_2^4 e_4
\\
&\qquad\qquad\qquad\qquad\qquad
-12 e_1^3 e_2 e_3 e_4+18 e_1^2 e_3^2 e_4+3 e_1^4 e_4^2+9 e_2^2 e_4^2-18 e_1 e_3 e_4^2+2 e_4^3
\\
P_4^{(6,3)}&=   e_3^6-6 e_2 e_3^4 e_4+9 e_2^2 e_3^2 e_4^2+6 e_1 e_3^3 e_4^2-2 e_2^3 e_4^3-12 e_1 e_2 e_3 e_4^3-6 e_3^2 e_4^3+3 e_1^2 e_4^4+6 e_2 e_4^4
\\
P_4^{(6,4)}&=   e_4^6.
\end{align*}

We are interested in the Lie group $\SLd$ rather than $\GLd$. For a matrix $A\in \SLd(\mathbb K)$  with eigenvalues $\lambda_1$, $\lambda_2$, \dots, $\lambda_d$ in the algebraic closure of the field $\mathbb K$, we have that
$$
E_d^{(d)} (\lambda_1, \lambda_2, \dots, \lambda_d) = \lambda_1 \lambda_2 \dots \lambda_d = \det A =1.
$$
It is therefore natural to specialize the polynomial $P_d^{(n,i)} \in \Z[e_1, e_2, \dots, e_d]$ by setting $e_d=1$, and to consider the \emph{reduced power elementary   polynomial} $\widehat P_d^{(n,i)} \in \Z[e_1, e_2, \dots, e_{d-1}]$ defined by
$$
\widehat P_d^{(n,i)} (e_1, e_2, \dots, e_{d-1}) = P_d^{(n,i)} (e_1, e_2, \dots, e_{d-1}, 1). 
$$

\begin{lem}
 The power elementary   polynomial $P_d^{(n,i)}$ is the unique polynomial  in $ \Z[e_1, e_2, \dots, e_d]$ such that
 $$
 E_d^{(i)}(A^n) =  P_d^{(n,i)} \left( E_d^{(1)}(A), E_d^{(2)}(A), \dots, E_{(d)}^d(A) \right)
 $$
 for every $A\in \GLd$.
 
 The reduced power elementary   polynomial $\widehat P_d^{(n,i)}$ is the unique polynomial  $\Z[e_1, e_2, \dots, e_{d-1}]$ such that
 $$
 E_d^{(i)}(A^n) = \widehat P_d^{(n,i)} \left( E_d^{(1)}(A), E_d^{(2)}(A), \dots, E_d^{(d-1)}(A) \right)
 $$
 for every $A\in \SLd$.
\end{lem}
\begin{proof} If a matrix $A\in \GLd$ has eigenvalues $\lambda_1$, $\lambda_2$, \dots, $\lambda_d$, its $n$--th power $A^n$ has eigenvalues $\lambda_1^n$, $\lambda_2^n$, \dots, $\lambda_d^n$. 
 The fact that $P_d^{(n,i)}$ and  $\widehat P_d^{(n,i)}$  satisfy the relations indicated  then follows from their definitions, noting that $E_d^{(d)}(A)=1$ for every $A\in \SLd$. The uniqueness property immediately follows from the fact that the polynomials $E_d^{(1)}$, $E_d^{(2)}$, \dots, $E_d^{(d)}$ are algebraically independent in $\Z[\lambda_1, \lambda_2, \dots, \lambda_d]$ (see \cite[\S I.2]{McD}). 
\end{proof}

For future reference, we note the following elementary homogeneity property of the power elementary   polynomials $P_d^{(n,ki)} \in \Z[e_1, e_2, \dots, e_d]$.

\begin{lem}
\label{lem:PowerElementaryPolHomogeneous}
 For every scalar $\theta \in \C$, 
 $$
 P_d^{(n,i)} (\theta e_1, \theta^2 e_2, \dots, \theta^d e_d) =  \theta^{ni} P_d^{(n,i)} ( e_1,  e_2, \dots,  e_d).
 $$
\end{lem}

\begin{proof}
 This is an immediate consequence of the property that each elementary symmetric polynomial $E_d^{(i)}\in \Z[\lambda_1, \lambda_2, \dots, \lambda_d]$ is homogeneous of degree $i$, while the power elementary symmetric polynomial $E_d^{(n,i)} $ is homogeneous of degree $ni$. 
\end{proof}

The following result is much less natural, but it will play an essential role in the proof of the main result of this article.

\begin{prop}
\label{prop:PowerElementaryPolMainProp}
Given commuting variables $x_1$, $x_2$, \dots, $x_{d-1}$ with $x_{d-1}$ invertible,  define
$$
y_j =
\begin{cases}
x_{d-1}^{-1}  + x_1 &\text{if } j=1
 \\
 x_{j-1}x_{d-1}^{-1} + x_j &\text{if } 2\leq j \leq d-1.
\end{cases}
$$
Then, for every $n$ and every $i$ with $1\leq i\leq d-1$, we have the following equality
$$
\widehat P_d^{(n,i)}(y_1, y_2, \dots, y_{d-1}) = x_{d-1}^{-n}P_{d-1}^{(n, i-1)} (x_1, x_2, \dots, x_{d-1}) +  P_{d-1}^{(n, i)} (x_1, x_2, \dots, x_{d-1})
$$
 of Laurent polynomials in $\Z[x_1, x_2, \dots, x_{d-2}, x_{d-1}^{\pm1}]$. 
\end{prop}

\begin{proof} The proof should make the statement less mysterious. 
 Consider the ring homomorphism
 $$
\phi \colon  \Z[x_1, x_2, \dots, x_{d-2}, x_{d-1}^{\pm1}] \to \Z [\lambda_1, \lambda_2, \dots, \lambda_d]/(\lambda_1\lambda_2 \dots \lambda_d=1)
 $$
 sending each $x_j$ to the elementary symmetric polynomial 
 $E_{d-1}^{(j)} \in \Z[\lambda_1, \lambda_2, \dots, \lambda_{d-1}]$ in the first $d-1$ variables, 
 and sending $x_{d-1}^{-1}$ to $\lambda_d$. Note that $\phi$ is well-defined since, in the target space, 
 $$\phi(x_{d-1}^{-1}) = \lambda_d = (\lambda_1 \lambda_2 \dots \lambda_{d-1})^{-1} = (E_{d-1}^{(d-1)})^{-1} = \phi(x_{d-1})^{-1}.$$
 
 Using the fact that the $E_{d-1}^{(i)}$ are algebraically independent in $ \Z[\lambda_1, \lambda_2, \dots, \lambda_{d-1}]$, a simple argument shows that $\phi$ is injective. To prove the proposed relation, we therefore only need to show that the two sides have the same image under $\phi$. 
 
 The key property underlying the whole result is that, for $2\leq j \leq d-1$,
\begin{align*}
 \phi(y_j) 
 &= \phi(x_{j-1}x_{d-1}^{-1} + x_j)=  E_{d-1}^{(j-1)} \lambda_d + E_{d-1}^{(j)}
 \\
 &= \lambda_d \sum_{1\leq i_1< \dots<i_{j-1}\leq d-1} \lambda_{i_1} \lambda_{i_2} \dots \lambda_{i_{j-1}} 
 +  \sum_{1\leq i_1< \dots<i_j \leq d-1} \lambda_{i_1} \lambda_{i_2} \dots \lambda_{i_j} 
 \\
 &=   \sum_{1\leq i_1< \dots<i_j \leq d} \lambda_{i_1} \lambda_{i_2} \dots \lambda_{i_j}  = E_d^{(j)}. 
\end{align*}
A similar argument shows that $\phi(y_1)=E_d^{(1)}$. 

Then, for the left-hand side of the proposed equality,
\begin{align*}
 \phi \big( \widehat P_d^{(n,i)}(y_1, y_2, \dots, y_{d-1})  \big) &= \widehat P_d^{(n,i)}\big( \phi(y_1), \phi(y_2), \dots, \phi(y_{d-1}) \big)
 \\
 &= \widehat P_d^{(n,i)} (E_d^{(1)}, E_d^{(2)}, \dots, E_d^{(d-1)}) 
 \\
 &=  P_d^{(n,i)} (E_d^{(1)}, E_d^{(2)}, \dots, E_d^{(d-1)}, 1)
 \\
 &=  P_d^{(n,i)} (E_d^{(1)}, E_d^{(2)}, \dots, E_d^{(d-1)}, E_d^{(d)}) = E_d^{(n,i)}
\end{align*}
using the property that $E_d^{(d)} = \lambda_1 \lambda_2 \dots \lambda_d=1$ in the target space of $\phi$.

 For the right-hand side,
\begin{align*}
 \phi \big( x_{d-1}^{-n}  P_{d-1}^{(n, i-1)}& (x_1, x_2, \dots, x_{d-1}) +  P_{d-1}^{(n, i)} (x_1, x_2, \dots, x_{d-1}) \big) 
 \\
 &=\phi(x_{d-1}^{-1})^n P_{d-1}^{(n, i-1)} \big( \phi(x_1), \phi(x_2), \dots, \phi(x_{d-1}) \big) 
 \\
 &\qquad\qquad\qquad\qquad\qquad
 +  P_{d-1}^{(n, i)} \big( \phi(x_1), \phi(x_2), \dots, \phi(x_{d-1}) \big) 
 \\
 &= \lambda_d^n P_{d-1}^{(n, i-1)} (E_{d-1}^{(1)}, E_{d-1}^{(2)}, \dots, E_{d-1}^{(d-1)}) + P_{d-1}^{(n, i)} (E_{d-1}^{(1)}, E_{d-1}^{(2)}, \dots, E_{d-1}^{(d-1)}) 
 \\
 &= \lambda_d^n E_{d-1}^{(n, i-1)}+ E_{d-1}^{(n, i)} 
 \\
 &= \lambda_d^n \sum_{1\leq j_1< \dots<j_{i-1}\leq d-1} \lambda_{j_1}^n \lambda_{j_2}^n \dots \lambda_{j_{i-1}}^n 
 +  \sum_{1\leq j_1< \dots<j_i\leq d-1} \lambda_{j_1}^n \lambda_{j_2}^n \dots \lambda_{j_i}^n
 \\
 &=   \sum_{1\leq j_1< \dots<j_i\leq d} \lambda_{j_1}^n \lambda_{j_2}^n \dots \lambda_{j_i}^n = E_d^{(n,i)} = \phi \big( \widehat P_d^{(n,i)}(y_1, y_2, \dots, y_{d-1})  \big).  
\end{align*}

Since $\phi$ is injective, this concludes the proof. 
\end{proof}

\section{Computations in the annulus}
\label{sect:AnnulusComputations}

Inspired by earlier constructions of  Morton \cite{Mort}, L\^e \cite{Le1} and Queffelec-Wedrich \cite{QueWed}, we let $A= S^1 \times [0,1] $ be the annulus with two marked points $x_0 = (x,0)$ and $x_1=(x,1)$ on the boundary (for an arbitrary $x\in S^1$). Let $\Aio$ be the vector space generated by webs in $A$ with boundary $\{ x_0, x_1\}$, with orientation going inward at $x_0$ and  outward at $x_1$, and quotiented by the skein relations of \cite{CKM}. (The subscript $\mathrm{io}$  stands for ``in-out''.)

Figure~\ref{fig:Webs1} offers a few examples of webs representing elements of $\Aio$ in $A$. In particular, let $I\in \Aio$ be represented by the arc $x\times[0,1]$ of the first picture of Figure~\ref{fig:Webs1}, endowed with weight 1, and let the \emph{twist element} $T\in \Aio$ be the arc of the second diagram if Figure~\ref{fig:Webs1}, also endowed with weight 1. A more elaborate element $X_j \in \Aio$, with $1\leq j \leq d-2$, is represented by the third web of Figure~\ref{fig:Webs1}. 

\begin{figure}[htbp]

\SetLabels
(  .15*   -.15 )  $I$ \\
( .5 * -.15 )  $T$ \\
( .85 *   -.15 )   $X_j$ \\
( .17*.8 ) $1$ \\
(.5 * .1) $1$ \\
\L( .87*.83 ) $j+1$ \\
(.85 * .1) $j$ \\
( .89* .58) $1$ \\
( .78* .83) $1$ \\
\endSetLabels
\centerline{\AffixLabels{\includegraphics{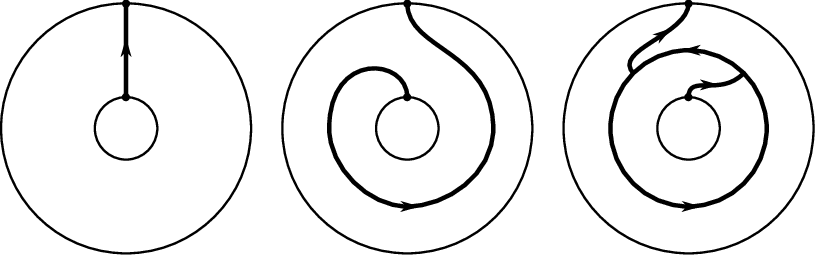}}}
\vskip 20pt

\caption{A few webs, representing the elements $I$, $T$, $X_j \in \Aio$}
\label{fig:Webs1}
\end{figure}

The  space $\Aio$ comes with a multiplication
$$
\circ\ \colon \Aio \otimes \Aio \to \Aio
$$
where the skein $W_1 \circ W_2$ is defined by placing $W_1$ in $S^1 \times [0,\frac12]$ and $W_2$ in $S^1\times [\frac12, 1]$. 

It also comes with a left and a right action of the usual skein algebra $\SSS(A)$ where, if $[W_0] \in \SSS(A)$ and $[W_1]\in \Aio$, $[W_0] \bullet  [W_1]$ is obtained by placing $[W_0]$ below $[W_1]$ and $[W_1] \bullet  [W_0]$ is obtained by placing $[W_0]$ on top of   $[W_1]$. We are particularly interested in the elements $I \bullet  L_j$ and $L_j \bullet  I\in \Aio$ illustrated in the last two pictures of Figure~\ref{fig:Webs2}, where $L_j \in \SSS(A)$ is represented by a simple loop going counterclockwise around the annulus and carrying weight $j$. 

\begin{figure}[htbp]

\SetLabels
(  .15*   -.15 )  $L_j$ \\
( .5 * -.15 )  $I \bullet  L_j$ \\
( .85 *   -.15 )   $L_j \bullet  I$ \\
( .15*.1 ) $j$ \\
(.5 * .1) $j$ \\
(.85 * .1) $j$ \\
( .485* .67) $1$ \\
( .83 * .67) $1$ \\
\endSetLabels
\centerline{\AffixLabels{\includegraphics{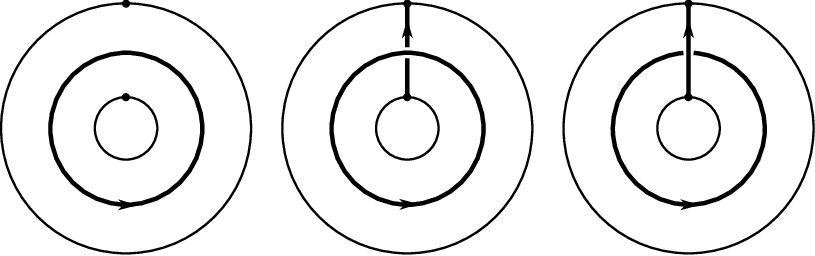}}}
\vskip 20pt

\caption{The skeins $L_j \in \SSS(A)$, $I \bullet  L_j\in \Aio$ and $L_j \bullet  I\in \Aio$}
\label{fig:Webs2}
\end{figure}

The following  lemma states that the elements $I$, $T$, $I \bullet  L_j$ and $L_j \bullet  I$ are central in $\Aio$, for the multiplication $\circ$. 

\begin{lem}
\label{lem:CentralElementsInAio}
 For every $X\in \Aio$,
\begin{align*}
 X\circ I &= I \circ X = X
 &
 X \circ T &= T\circ X
 \\
 X \circ (I \bullet  L_j) &= (I \bullet  L_j) \circ X
 &
 X \circ ( L_j \bullet  I) &= ( L_j \bullet  I) \circ X.
\end{align*}
\end{lem}
\begin{proof}
 These properties are easily checked by elementary isotopies in the thickened annulus $A\times [0,1]$. 
\end{proof}

A less immediate relation between the skeins of Figures~\ref{fig:Webs1}--\ref{fig:Webs2} is provided by the skein relations of \S \ref{sect:SkeinModule}.  

\begin{lem}
\label{lem:ITimesL_j}
For $1\leq j \leq d-1$, 
\begin{align*}
 I \bullet  L_j&= 
\begin{cases}
q^{\frac {d-1}d}  T  -q^{-\frac 1d }   X_1 & \text{if } j=1
 \\
 (-1)^{j-1} q^{\frac {d-j}d}  X_{k-1} \circ T + (-1)^j q^{-\frac jd} X_j \hskip 1em & \text{if } 2\leq j \leq d-2
 \\
  (-1)^{d-2} q^{\frac {1}d}  X_{d-2} \circ T + q^{\frac {1-d}d} T^{-1}& \text{if } j=d-1
  \end{cases}
  \\
   L_j \bullet  I  &= 
\begin{cases}
q^{\frac {1-d}d}  T  -q^{\frac 1d }   X_1 & \hskip .5em \text{if } j=1
 \\
 (-1)^{j-1} q^{\frac {j-d}d}  X_{j-1} \circ T + (-1)^j q^{\frac jd} X_j \hskip 1em & \hskip .5em  \text{if }  2\leq j \leq d-2
 \\
  (-1)^{d-2} q^{-\frac {1}d}  X_{d-2} \circ T +  q^{\frac {d-1}d} T^{-1}& \hskip .5em \text{if } j=d-1
\end{cases}
\end{align*}
where $T^{-1}$ is the inverse of $T$ for the composition operation $\circ$ (which is also its mirror image). 
\end{lem}
\begin{proof}
 This follows from an application of the braiding relations of Figure~\ref{fig:Braiding}, which express $ L_j \bullet  I $ and $ I \bullet  L_j$ as a linear combination of two webs.

  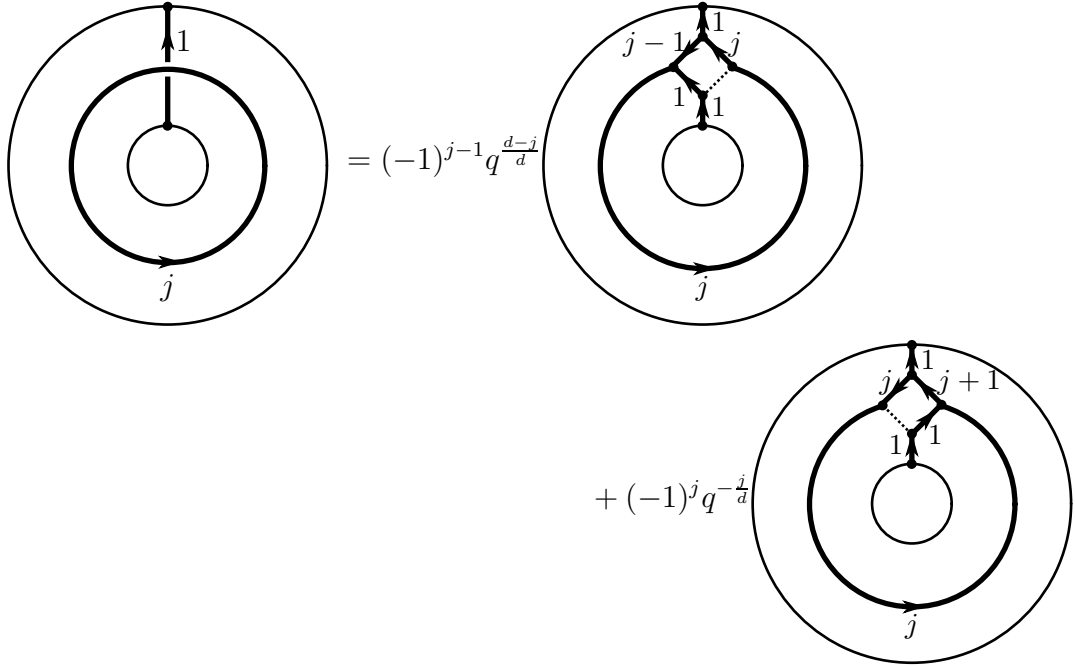
\begin{figure}[htbp]
\begin{align*}
 \raisebox{-60pt}{\SetLabels
( .55 * .85)  $1$ \\
( .5 * .1 )  $j$ \\
(  *  )   \\
\endSetLabels
{\AffixLabels{\includegraphics{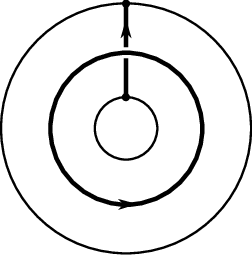}}}
} &= 
 (-1)^{j-1} q^{\frac {d-j}d} 
 \raisebox{-60pt}{\SetLabels\small
( .55 * .91)  $1$ \\
( .55 * .65)  $1$ \\
( .5 * .1 )  $j$ \\
\L( .58 * .85 )  $j$ \\
\R(  .45* .85 )  $j-1$ \\
\R( .45 * .68 )   $1$\\
\endSetLabels
{\AffixLabels{\includegraphics{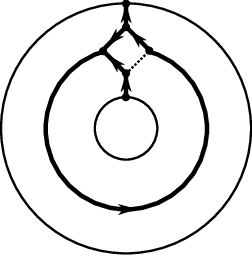}}}
}
\\
&\qquad\qquad\qquad\qquad
 + (-1)^j q^{-\frac jd}
\raisebox{-60pt}{\SetLabels\small
( .55 * .91)  $1$ \\
( .45 * .65)  $1$ \\
( .5 * .1 )  $j$ \\
\L( .58 * .85 )  $j+1$ \\
\R(  .44* .85 )  $j$ \\
\L( .55 * .68 )   $1$\\
\endSetLabels
{\AffixLabels{\includegraphics{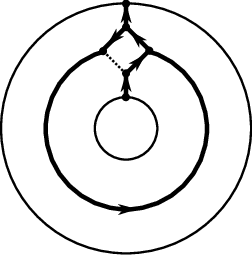}}}
}
\end{align*}

\caption{The proof of Lemma~\ref{lem:ITimesL_j} when $2\leq j \leq d-2$}
\label{fig:ITimesL_j}
\end{figure}
 
 When $2\leq j \leq d-2$, the computation for $I \bullet L_j $ is illustrated in Figure~\ref{fig:ITimesL_j}. On the right hand side of the equation, the webs represented each have one edge carrying weight $0$ (represented by a dotted line in the pictures) which must be erased by the conventions stated in \S \ref{sect:SkeinModule}. The first web is easily seen to be isotopic to $X_{k-1} \circ T$, while the second web is isotopic to $X_j$.

 When $j=1$, the first web occurring in the same computation now has two edge weights equal to 0. After erasing the corresponding two edges, the resulting web is isotopic to $T$. The second web is still isotopic to $X_1$. 
 
  \begin{figure}[htbp]

\begin{align*}
\raisebox{-60pt}{\SetLabels\small
( .55 * .91)  $1$ \\
( .45 * .65)  $1$ \\
( .5 * .1 )  $d-1$ \\
\L( .58 * .85 )  $d$ \\
\R(  .44* .85 )  $d-1$ \\
\L( .55 * .68 )   $1$\\
\endSetLabels
{\AffixLabels{\includegraphics{Webs5.eps}}}
}
=
\raisebox{-60pt}{\SetLabels\small
( .55 * .91)  $1$ \\
( .45 * .65)  $1$ \\
( .5 * .1 )  $d-1$ \\
\L( .63 * .85 )  $d$ \\
\R(  .38* .89 )  $d$ \\
\L( .55 * .68 )   $1$\\
\endSetLabels
{\AffixLabels{\includegraphics{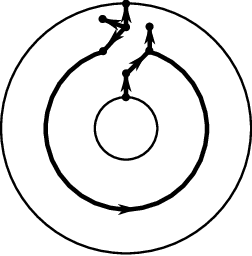}}}
}
=(-1)^{d-1}
\raisebox{-60pt}{\SetLabels\small
( .55 * .91)  $1$ \\
( .45 * .65)  $1$ \\
( .5 * .1 )  $1$ \\
\L( .55 * .68 )   $1$\\
\endSetLabels
{\AffixLabels{\includegraphics{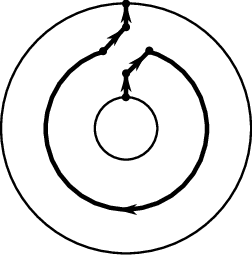}}}
}
\end{align*}
\caption{The proof of Lemma~\ref{lem:ITimesL_j} when $ k = d-1$}
\label{fig:ITimesL_(d-1)}
\end{figure}

 When $j=d-1$, the first web is still $X_{d-2} \circ T$ but the second web has an edge weight equal to $d$. We now need to use the conventions of \cite{CKM} for this case, which we had skipped in our discussion in  \S \ref{sect:SkeinModule}. These  involve a two-step process, first splitting the weight $d$ edge into two stumps and then flipping the resulting inward stump to the other side of the split vertex at which it is attached (see the top of Page 358 of \cite{CKM}).  After applying the second and third skein relations of Figure~\ref{fig:Skein3} followed by an isotopy, we obtain the mirror image of $T$, which is also $T^{-1}$ in $\Aio$. See Figure~\ref{fig:ITimesL_(d-1)}.  
 
 This completes the proof of the statement of Lemma~\ref{lem:ITimesL_j} for $I \bullet  L_j$. The proof for  $L_j \bullet I$ is essentially identical. 
\end{proof}

\begin{lem}
\label{lem:X_jCentralInAio}
  For every $X\in \Aio$ and every $j$ with $1\leq j \leq d-2$, 
  $$
  X \circ X_j = X_j \circ X.
  $$
\end{lem}
\begin{proof}
 By induction on $j$, the formulas of Lemma~\ref{lem:ITimesL_j} show that, for the multiplication by composition $\circ$, the skein $X_j \in \Aio$ can be expressed as a polynomial in the skeins $I$, $T$ and $ L_j \bullet  I  $. Since these skeins are central in $\Aio$ by Lemma~\ref{lem:CentralElementsInAio}, so is $X_j$. 
\end{proof}

\begin{prop}
\label{prop:PowerElementaryTimesI}
Suppose that the $d$--root $q^{\frac1d}$ occurring in the braiding relations of Figure~{\upshape\ref{fig:Braiding}} is a $2n$--root of unity, and let $\widehat P_d^{(n,i)} \in \Z[e_1, e_2, \dots, e_{d-1}]$ be the reduced power elementary   polynomial of {\upshape \S \ref{sect:PowerPols}}. Then, for the framed link $L\subset A$ and the skein $I \in \Aio$ as above,
$$
L^{[\widehat P_d^{(n,i)}]} \bullet  I = I \bullet L^{[\widehat P_d^{(n,i)}]} .
$$
\end{prop}

\begin{proof}
 Consider $\Aio$ as a ring for the multiplication by composition $\circ$. Then, the commutativity property of Lemma~\ref{lem:X_jCentralInAio} shows that there is a unique ring homomorphism 
 $$
 \psi \colon \Z [x_1, x_2, \dots, x_{d-2}, x_{d-1}^{\pm1}] \to \Aio
 $$
 such that $\psi(x_{d-1})= q^{\frac{1-d}d}T^{-1}$ and $\psi(x_j) = (-1)^j q^{\frac jd} X_j$ for every $j$ with  with $1\leq j \leq d-2$. 
 
 If we set 
 $$
y_j =
\begin{cases}
x_{d-1}^{-1} + x_1 &\text{if } i=1
 \\
 x_{j-1} x_{d-1}^{-1} + x_j &\text{if } 2\leq j \leq d-1 
\end{cases}
$$
as in Proposition~\ref{prop:PowerElementaryPolMainProp}, the first batch of computations in Lemma~\ref{lem:ITimesL_j} show that $\psi(y_j)= L_j \bullet  I $ for every $j$. Applying the ring homomorphism $\psi$  to both sides of the conclusion 
$$
\widehat P_d^{(n,i)}(y_1, y_2, \dots, y_{d-1}) = x_{d-1}^{-n}P_{d-1}^{(n, i-1)} (x_1, x_2, \dots, x_{d-1}) +  P_{d-1}^{(n, i)} (x_1, x_2, \dots, x_{d-1})
$$
of Proposition~\ref{prop:PowerElementaryPolMainProp}, we conclude that 
\begin{align*}
 \widehat P_d^{(n,i)}( L_1 \bullet  I,  L_2 \bullet  I, &\dots,  L_{d-1}\bullet  I)
 \\
  &= q^{\frac {n(1-d)}d} T^n \circ  P_{d-1}^{(n, i-1)} \big( -q^{\frac1d} X_1,  +q^{\frac2d} X_2, \dots,  (-1)^{d-1} q^{\frac{d-1}d}X_{d-1} \big) 
 \\
 &\qquad\qquad +  P_{d-1}^{(n, i)} \big( -q^{\frac1d} X_1,  +q^{\frac2d} X_2, \dots,  (-1)^{d-1} q^{\frac{d-1}d}X_{d-1} \big) 
 \\
 &= (-1)^{n(i-1)} q^{\frac {n(i+1-d)}d} \, T^n \circ  P_{d-1}^{(n, i-1)} \big(  X_1,   X_2, \dots,  X_{d-1} \big) 
 \\
 &\qquad\qquad +   (-1)^{ni} q^{\frac {ni}d}  P_{d-1}^{(n, i)} \big( X_1,  X_2, \dots,  X_{d-1} \big),
\end{align*}
using Lemma~\ref{lem:PowerElementaryPolHomogeneous} for the second equality.

When evaluating a polynomial on elements of $\Aio$, we used the multiplication by juxtaposition $\circ$. However, in the case of the skeins $L_j \bullet  I$, this evaluation is also closely related to the multiplication by superposition $\bullet$ and to the threading operation. Indeed, by inspection of the definitions, 
$$
P( L_1 \bullet  I,  L_2 \bullet  I, \dots,  L_{d-1} \bullet  I) = L^{[P]} \bullet I
$$
for every polynomial $P \in \Z[e_1, e_2, \dots, e_{d-1}]$. In particular, we now conclude that 
\begin{align*}
 L^{[\widehat P_d^{(n,i)}]} \bullet I &= (-1)^{n(i-1)} q^{\frac {n(i+1-d)}d} \, T^n \circ  P_{d-1}^{(n, i-1)} \big(  X_1,   X_2, \dots,  X_{d-1} \big) 
 \\
 &\qquad\qquad\qquad\qquad +   (-1)^{ni} q^{\frac {ni}d} \, P_{d-1}^{(n, i)} \big( X_1,  X_2, \dots,  X_{d-1} \big).
\end{align*}

If we now use the second batch of computations in Lemma~\ref{lem:ITimesL_j}, where $q$ is replaced by $q^{-1}$, the same arguments show that
\begin{align*}
I \bullet  L^{[\widehat P_d^{n,k}]}   &= (-1)^{n(i-1)} q^{-\frac {n(i+1-d)}d} \, T^n \circ  P_{d-1}^{(n, i-1)} \big(  X_1,   X_2, \dots,  X_{d-1} \big) 
 \\
 &\qquad\qquad\qquad\qquad +   (-1)^{ni} q^{-\frac {ni}d} \, P_{d-1}^{(n, i)} \big( X_1,  X_2, \dots,  X_{d-1} \big).
\end{align*}

We are now ready to use our hypothesis that $q^{\frac1d}$ is a $2n$--root of unity, which means that $ q^{\frac nd}  =   q^{-\frac nd} $. The above computations then show that $ L^{[\widehat P_d^{(n,i)}]} \bullet I  = I \bullet  L^{[\widehat P_d^{(n,i)}]}$. 
\end{proof}

\section{Central and transparent skeins from power elementary  polynomials}
\label{sect:CentralTransparent}

We now use Proposition~\ref{prop:PowerElementaryTimesI} to construct transparent elements in the skein module $\SSS(M)$. The following lemma will enable us to limit our argument to web edges that carry weight 1. 

\begin{lem}
\label{lem:Explosion}
 Let $W$ be a web in the $3$--manifold $M$, and let $B\subset M$ be a ball meeting $W$ along an arc contained in the interior of an edge $e$ of $W$ carrying weight $i \in \{1,2,\dots, d-1\}$. Suppose that $q^{2j}\neq 1$ for every $j \in \{2,3,\dots, i\}$, so that the quantum factorial $\qfact{i}$ is nonzero. Then there exists a web $W'$ in $M$ such that
\begin{enumerate}
 \item $[W] = \frac1{\qfact{i}} [W']$ in the skein module $\SSS(M)$;
 \item every point of $B\cap W'$ is contained in a weight $1$ edge of $W'$;
\item $W'$ is contained in an arbitrarily small neighborhood of $W$. 
\end{enumerate}
\end{lem}

\begin{proof} The skein relation of Figure~\ref{fig:Skein2} gives us the relation of Figure~\ref{fig:Explosion}. 
The result then follows by repeated application of this property. 
\end{proof}

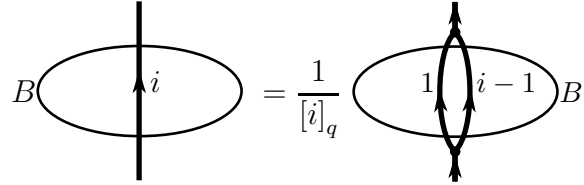
\begin{figure}[htbp]

\begin{align*}
 \raisebox{-30pt}{\SetLabels
\L( .55 * .5 ) $i$  \\
\R\E(  0* .5 )  $B$ \\
(  *  )   \\
\endSetLabels
{\AffixLabels{\includegraphics{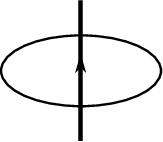}}}
} 
= \frac1{\qint i}\ 
  \raisebox{-30pt}{\SetLabels\small
\L( .6 * .5 ) $i-1$  \\
\L\E(  1* .5 )  $B$ \\
\R( .4 * .5 )   $1$\\
\endSetLabels
{\AffixLabels{\includegraphics{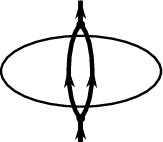}}}
} 
\end{align*}
\caption{The proof of Lemma~\ref{lem:Explosion}}
\label{fig:Explosion}
\end{figure}

\begin{thm}
\label{thm:PowerElementaryThreadingTransparent}
Suppose that the $d$--root $q^{\frac 1d}$ occurring in the definition of skein modules $\SSS(M)$ is such that $q^{\frac {2n}d}=1$, and that $q^{2i}\neq 1$ for every integer $i$ with $2\leq i \leq \frac d2$. Then, for every $i=1$, $2$, \dots, $d-1$,   threading  the reduced power elementary  polynomial  $\widehat P_d^{(n,i)} \in \Z[e_1, e_2, \dots, e_{d-1}]$ is transparent in the skein module $\SSS(M)$ of any oriented  $3$--manifold $M$. 
\end{thm}

\begin{proof}
 Let $W$ be a web in a oriented 3--manifold $M$, and let $L_1$ and $L_2$ be two framed knots in $M$ that are disjoint from $W$ and isotopic to each other by an isotopy that is allowed to cross $W$. We want to show that $L_1^{[\widehat P_d^{(n,i)}]} \sqcup W$ and $L_2^{[\widehat P_d^{(n,i)}]} \sqcup W$ represent the same element in $\SSS(M)$. 
 
By decomposing the isotopy into little steps, we can clearly restrict attention to the case where it crosses $W$ in exactly one point, located in an edge of $W$ carrying weight $i$. If $i>\frac d2$, we can use the second skein relation of Figure~\ref{fig:Skein3} to replace $i$ by $d-i$; we can therefore assume that $i \leq \frac d2$, and in particular that $\qfact i \neq 0$ by our hypotheses on $q$. Then, applying Lemma~\ref{lem:Explosion} to a small ball around the crossing point enables us to restrict attention to the case where the isotopy crosses $W$ transversely in one point contained in an edge $e$ with weight $1$. 

 By transversality, we can further choose the isotopy so that, for the annulus $A$, there is an embedding of $A\times [0,1]$ in $M$ such that:
\begin{enumerate}
\item the intersection of the edge $e$ with $A\times [0,1]$ is equal to $I \times \frac 12$ for the arc $I$ of Figure~\ref{fig:Webs1}, and the ribbon structure there is horizontal for the projection to $A$;

\item shortly around the time when the isotopy crosses $e$, the link is contained in $A\times [0,1]$, its projection to $A$ is equal to the link $L$ of Figure~\ref{fig:Webs2}, and its ribbon structure is horizontal;

\item the link is contained in $A\times [\frac12,1]$ shortly before the isotopy crosses $W$, and in $A\times [0,\frac12]$ shortly after that.
\end{enumerate}
Restricting the isotopy to times near the crossing time, we can even assume that $L_1$ is contained in $A\times [\frac12,1]$, that $L_2$ is contained in $A\times [0,\frac12]$

We can then apply Proposition~\ref{prop:PowerElementaryTimesI} to conclude that the intersections of $L_1^{[\widehat P_d^{(n,i)}]} \sqcup W$ and $L_2^{[\widehat P_d^{(n,i)}]} \sqcup W $  with $A\times [0,1]$ differ by a sequence of isotopies and skein relations supported in the interior of $A\times [0,1]$. Since $L_1^{[\widehat P_d^{(n,i)}]} \sqcup W$ and $L_2^{[\widehat P_d^{(n,i)}]} \sqcup W $ coincide outside of $A\times [0,1]$, we conclude that $[L_1^{[\widehat P_d^{(n,i)}]} \sqcup W] = [L_2^{[\widehat P_d^{(n,i)}]} \sqcup W]$ in $\SSS(M)$. 
\end{proof}

Applying Lemma~\ref{lem:TransparentGivesCentral}, an immediate corollary is that Theorem~\ref{thm:PowerElementaryThreadingTransparent} provides central elements in the skein algebra $\SSS(S)$.

\begin{cor}
 \label{cor:PowerElementaryThreadingCentral}
Suppose that the $d$--root $q^{\frac 1d}$ occurring in the definition of the skein algebra $\SSS(M)$ is such that $q^{\frac {2n}d}=1$, and that $q^{2i}\neq 1$ for every integer $i$ with $2\leq i \leq \frac d2$. In a thickened surface $S \times [0,1]$, let $L=L_1 \sqcup L_2 \sqcup \dots \sqcup L_c$ be a framed link in which each component $L_j$ carries a weight $i_j \in \{1,2, \dots, d-1\}$. Then the skein $L^{[\widehat P_d^{(n, \bullet)}]} \in \SSS(S)$ obtained by threading the reduced power elementary   polynomial $\widehat P_d^{(n, i_j)} \in \Z[e_1, e_2, \dots, e_{d-1}]$ along each component $L_j$ is central in the skein algebra $\SSS(S)$ of the surface $S$.  \qed
\end{cor}

\begin{rem}
\label{rem:ConditionsOnQ}
 In the statements of Theorem~\ref{thm:PowerElementaryThreadingTransparent} and Corollary~\ref{cor:PowerElementaryThreadingCentral}, the condition that $q^{2i}\neq 1$ for every $i$ with $2\leq i \leq \frac d2$ is an artifact of our use of Lemma~\ref{lem:Explosion} in the proof, and is probably unnecessary. 
\end{rem}
 
 \section{Two conjectures}
 \label{sect:Conjectures}
 
 We conclude with two conjectures. 
 
 The first conjecture is the obvious one regarding the center of the skein algebra $\SSS(S)$. In addition to the elements exhibited in this article, the center of $\SSS(S)$ admits more obvious  elements associated to the punctures of the surface $S$. Indeed, if $[P_i]\in \SSS(S)$ is represented by a small loop going around one of the punctures of the surface $S$, endowed with a weight $i \in \{1,2,\dots, d-1\}$, a simple isotopy shows that $[P_i]$ in central in $\SSS(S)$, and this for any value of $q$. 
  
\begin{con}
\label{con:WholeCenter}
Suppose that the $d$--root $q^{\frac 1d}$ occurring in the definition of the skein algebra $\SSS(S)$ is such that $q^{\frac 2d}$ is a primitive $n$--root of unity. Then, for every oriented surface $S$ of finite topological type,  the center of $\SSS(S)$ is generated (as a subalgebra) by the skeins $[P_i]$ associated to punctures as above, as well as by the skeins $L^{[\widehat P_d^{(n,i)}]}$ obtained by threading reduced power elementary   polynomials $\widehat P_d^{(n,i)}$ around framed knots $L \subset S\times [0,1]$. 
\end{con}

See \cite{FroKanBarLe} for a proof of this conjecture in the case where $d=2$.

 The second conjecture is the full $\SLd$ analogue of the main statement underlying the results of \cite{BonWon} for $\SSStwo(S)$. It essentially asserts that the central skeins $L^{[\widehat P_d^{(n,i)}]} \in \SSS(S)$, obtained by threading reduced power elementary polynomials along framed knots, satisfy the skein relations corresponding to $q^{\frac1d}=1$.

\begin{con}
\label{con:FrobeniusHomomorphism}
Let $S$ be an oriented surface of finite topological type.
 If the $d$--root $q^{\frac1d}$ occurring in the definition of the $\SLd$--skein algebra is a root of unity of order $n$ coprime with $2d$, 
 and if the commutative skein algebra $ \mathcal S_{\SLd}^1(S) $ is defined with the convention that $1^{\frac1d}=1$, 
there exists a  algebra homomorphism
 $$
 \Phi \colon \mathcal S_{\SLd}^1(S) \to \SSS(S)
 $$
 with central image 
 such that, for every skein $[L] \in  \mathcal S_{\SLd}^1(S)$ represented by a framed knot $L$ carrying weight $i \in \{ 1,2, \dots, d-1\}$, the image $\Phi\big( [L] \big) = L^{[\widehat P_d^{(n,i)}]}$ is obtained by threading the reduced power elementary   polynomial $\widehat P_d^{(n, i)} \in \Z[e_1, e_2, \dots, e_{d-1}]$ along $L$, in the sense defined in {\upshape\S \ref{sect:Thread}}. 
\end{con}

\begin{rem}
 It easily follows from the skein relations that the algebra $\mathcal S_{\SLd}^1(S)$ is generated by knots colored by a  weight $i \in \{ 1,2, \dots, d-1\}$. So the homomorphism $ \Phi \colon \mathcal S_{\SLd}^1(S) \to \SSS(S)$ is unique if it exists. 
\end{rem}

The case $d=2$ of this Conjecture~\ref{con:FrobeniusHomomorphism} was proved in \cite{BonWon} when $n$ is odd. See also \cite{BonWon, Le1} for related statements with other conditions on $q^{\frac12}$. These properties played a fundamental role in the study of the finite-dimensional representation theory \cite{BonWon, FroKanBarLe, GanJorSaf} of $\SSStwo(S)$. 

See \cite{Hig2} for a proof when $d=3$. 

For general $d$, the homomorphism predicted by Conjecture~\ref{con:FrobeniusHomomorphism} is likely to be the Frobenius homomorphism $\Phi \colon \mathcal S_{\SLthree}^1(S) \to \SSSthree(S)$ constructed for $d=3$ in \cite{Hig} (see also \cite{KorQue} for $d=2$), and conjectured to exist for all $d$. See \cite{Wan} for an explicit construction of this Frobenius homomorphism when the surface has nonempty boundary, and \cite{KimLe} for a related construction. Also see \cite{GanJorSaf, CosKorLe} for more general developments.

\bibliographystyle{amsalpha}
\bibliography{CentralElements}

\newcommand{\etalchar}[1]{$^{#1}$}
\providecommand{\bysame}{\leavevmode\hbox to3em{\hrulefill}\thinspace}
\providecommand{\MR}{\relax\ifhmode\unskip\space\fi MR }
\providecommand{\MRhref}[2]{%
  \href{http://www.ams.org/mathscinet-getitem?mr=#1}{#2}
}
\providecommand{\href}[2]{#2}
\begin{thebibliography}{BGBH{\etalchar{+}}23}

\bibitem[BGBH{\etalchar{+}}23]{BBHHMP}
Bodie Beaumont-Gould, Erik Brodsky, Alaina Hogan, Vijay Higgins, Joseph Melby,
  and Joshua Piazza, 2023, in preparation.

\bibitem[BW16]{BonWon}
Francis Bonahon and Helen Wong, \emph{Representations of the {K}auffman bracket
  skein algebra {I}: invariants and miraculous cancellations}, Invent. Math.
  \textbf{204} (2016), no.~1, 195--243. \MR{3480556}

\bibitem[CKL23]{CosKorLe}
Francesco Costantino, Julien Korinman, and Thang L\^{e}, 2023, in preparation.

\bibitem[CKM14]{CKM}
Sabin Cautis, Joel Kamnitzer, and Scott Morrison, \emph{Webs and quantum skew
  {H}owe duality}, Math. Ann. \textbf{360} (2014), no.~1-2, 351--390.
  \MR{3263166}

\bibitem[FKBL19]{FroKanBarLe}
Charles Frohman, Joanna Kania-Bartoszy{\'n}ska, and Thang L\^{e}, \emph{Unicity
  for representations of the {K}auffman bracket skein algebra}, Invent. Math.
  \textbf{215} (2019), no.~2, 609--650. \MR{3910071}

\bibitem[GJS19]{GanJorSaf}
Iordan Ganev, David Jordan, and Pavel Safronov, \emph{The quantum {F}robenius
  for character varieties and multiplicative quiver varieties}, 2019, to appear
  in J. Eur. Math. Soc., \texttt{arXiv.1901.11450}.

\bibitem[Hig23a]{Hig2}
Vijay Higgins, 2023, in preparation.

\bibitem[Hig23b]{Hig}
\bysame, \emph{Triangular decomposition of {$\mathrm{SL}_3$} skein algebras},
  Quantum Topol. \textbf{14} (2023), no.~1, 1--63. \MR{4609753}

\bibitem[KL23]{KimLe}
Hyun~Kyu Kim and Thang L\^{e}, 2023, in preparation.

\bibitem[KQ19]{KorQue}
Julien Korinman and Alexandre Quesney, \emph{Classical shadows of stated skein
  representations at roots of unity}, 2019, to appear in Alg. Geom. Topol.,
  \texttt{arXiv:1905.03441}.

\bibitem[Kup96]{Kup}
Greg Kuperberg, \emph{Spiders for rank {$2$} {L}ie algebras}, Comm. Math. Phys.
  \textbf{180} (1996), no.~1, 109--151. \MR{1403861}

\bibitem[L{\^e}15]{Le1}
Thang T.~Q. L{\^e}, \emph{On {K}auffman bracket skein modules at roots of
  unity}, Algebr. Geom. Topol. \textbf{15} (2015), no.~2, 1093--1117.
  \MR{3342686}

\bibitem[LS22]{LeSik}
Thang T.~Q. L{\^e} and Adam Sikora, \emph{Stated $\mathrm{SL}(n)$-skein modules
  and algebras}, 2022, preprint, \texttt{arXiv:2201.00045}.

\bibitem[Mac15]{McD}
Ian~G. Macdonald, \emph{Symmetric functions and {H}all polynomials}, second
  ed., Oxford Classic Texts in the Physical Sciences, The Clarendon Press,
  Oxford University Press, New York, 2015, With contribution by A. V.
  Zelevinsky and a foreword by Richard Stanley. \MR{3443860}

\bibitem[Mor02]{Mort}
Hugh~R. Morton, \emph{Skein theory and the {M}urphy operators}, J. Knot Theory
  Ramifications \textbf{11} (2002), no.~4, 475--492, Knots 2000 Korea, Vol. 2
  (Yongpyong). \MR{1915490}

\bibitem[Mor07]{Mor}
Scott Morrison, \emph{A diagrammatic category for the representation theory of
  $\mathrm u_q(\mathfrak{sl}_n)$}, 2007, PhD thesis, UC Berkeley,
  \texttt{arXiv:0704.1503}.

\bibitem[Pou22]{Pou}
Anup Poudel, \emph{A comparison between $\mathrm{SL}_n$ spider categories},
  2022, preprint, \texttt{arXiv:2210.09289}.

\bibitem[QW18]{QueWed}
Hoel Queffelec and Paul Wedrich, \emph{Extremal weight projectors {II}}, 2018,
  preprint, \texttt{arXiv:1803.09883}.

\bibitem[Sik05]{Sik}
Adam~S. Sikora, \emph{Skein theory for {${\rm SU}(n)$}-quantum invariants},
  Algebr. Geom. Topol. \textbf{5} (2005), 865--897. \MR{2171796}

\bibitem[Wan23]{Wan}
Zhihao Wang, \emph{On stated $\mathrm{SL}(n)$--skein modules}, 2023, preprint,
  \texttt{arXiv:2307.10288}.

\bibitem[Wit89]{Wit}
Edward Witten, \emph{Quantum field theory and the {J}ones polynomial}, Comm.
  Math. Phys. \textbf{121} (1989), no.~3, 351--399. \MR{990772}

\bibitem[Yok97]{Yok}
Yoshiyuki Yokota, \emph{Skeins and quantum {${\rm SU}(N)$} invariants of
  {$3$}-manifolds}, Math. Ann. \textbf{307} (1997), no.~1, 109--138.
  \MR{1427678}

\end{thebibliography}

\end{document}